\documentclass[a4paper]{amsart}
\usepackage{fullpage,amssymb,mathtools,verbatim}
\newtheorem{theorem}{Theorem}
\newtheorem{lemma}[theorem]{Lemma}
\newtheorem{proposition}[theorem]{Proposition}
\theoremstyle{remark}
\newtheorem{remarks}[theorem]{Remarks}

\numberwithin{theorem}{section}
\numberwithin{equation}{section}

\newcommand{\R}{\mathbb{R}}
\newcommand{\Q}{\mathbb{Q}}
\newcommand{\Z}{\mathbb{Z}}
\newcommand{\N}{\mathbb{N}}
\newcommand{\RP}{\mathbb{R}\mathbb{P}}
\newcommand{\p}{\mathfrak{p}}
\newcommand{\OK}{\mathcal{O}_K}

\DeclareMathOperator{\GL}{GL}
\DeclareMathOperator{\Disc}{Disc}
\DeclareMathOperator{\Norm}{Norm}
\DeclareMathOperator{\lcm}{lcm}
\begin{document}
\title{Sign patterns of Fourier coefficients of modular forms}
\author{Andrew R. Booker}
\address{School of Mathematics, University of Bristol,
Woodland Road, Bristol, BS8 1UG}
\email{andrew.booker@bristol.ac.uk}
\begin{abstract}
We give conditions under which a self-dual holomorphic cusp form
is determined up to scalar multiplication
by the signs of its Fourier coefficients.
\end{abstract}
\maketitle
\section{Introduction}
Let $f\in S_k(\Gamma_0(M))$ and $g\in S_\ell(\Gamma_0(N))$ be modular
forms with real Fourier coefficients $a_f(n)$ and $a_g(n)$. Various
authors have investigated the extent to which $f$ and $g$ are distinguished
by the signs of $a_f(n)$ and $a_g(n)$; see for instance
\cite{KLSW,Matomaki,GKR} and the references therein.
For Hecke eigenforms,
it follows from the recent proof of joint Sato--Tate by Wong \cite{Wong}
that if $f$ and $g$ are distinct, non-CM, normalized newforms
then each of the sets
$\{p\text{ prime}:\epsilon a_f(p)a_g(p)>0\}$
for $\epsilon\in\{\pm1\}$
has natural density $\frac12$.
In another direction, Gun, Kohnen, and Rath \cite{GKR} considered
forms $f$ and $g$ of distinct weights that are not necessarily
Hecke eigenforms; it follows from their result that if $k\ne\ell$ and
$a_f(n)a_g(n)$ is not identically zero\footnote{One can
find examples of non-zero $f$ and $g$ of distinct weights such that
$a_f(n)a_g(n)=0$ identically, so this hypothesis cannot be removed.}
then each of the sets $\{n\in\N:\epsilon a_f(n)a_g(n)>0\}$ is infinite.

In both of these results the forms in question are orthogonal,
so it is natural to expect sign changes to occur.
Without orthogonality the situation is much less clear, but
one might guess that the signs of $a_f(n)$ determine $f$ up to
scalar multiplication. That is not always the case (see the remarks below),
but we show that it is true in many cases:
\begin{theorem}\label{thm:main}
Let $M,N\in\N$, with $\lcm(M,N)$ not
divisible by $2^4$ or the square of an odd prime.
Let $f\in S_k^{\rm new}(\Gamma_0(M))$ and $g\in S_\ell^{\rm new}(\Gamma_0(N))$
be non-zero cusp forms with real Fourier coefficients
$a_f(n),a_g(n)$. Then the following are equivalent:
\begin{enumerate}
\item $a_f(n)a_g(n)\ge0$ for all sufficiently large $n\in\N$;
\item $(M,k)=(N,\ell)$ and $f/g$ is a positive constant.
\end{enumerate}
Thus if $f$ and $g$ are not proportional then
$\{n\in\N:\epsilon a_f(n)a_g(n)>0\}$
is infinite for each $\epsilon\in\{\pm1\}$.
\end{theorem}

\begin{remarks}\
\begin{enumerate}
\item The restriction on $\lcm(M,N)$
is necessary when $(M,k)=(N,\ell)$ and $Mk$ is sufficiently large,
since otherwise one can
choose a fundamental discriminant $\Delta\ne1$ with $\Delta^2\mid M$
and a twist-minimal, non-CM newform $f$ of conductor $M$.
Setting $g=2f+f\times\left(\frac{\Delta}{\cdot}\right)$, we have
$a_f(n)a_g(n)=a_f(n)^2(2+\left(\frac{\Delta}{n}\right))\ge0$.
\item The restriction to the new subspaces of
$S_k(\Gamma_0(M))$ and $S_\ell(\Gamma_0(N))$ is also
necessary in some cases, even under the assumption that $M=N$.
For example, let $f$ be the newform associated to an
elliptic curve $E$ of squarefree conductor $N_E$, and set $g(z)=f(z)+f(pz)$,
where $p>3$ is a supersingular prime for $E$. Then
$f,g\in S_2(\Gamma_0(pN_E))$, and $a_f(n)a_g(n)=a_f(n)^2\ge0$.
\item It will be clear from the proof that condition (1) of the theorem
can be substantially weakened. For instance, it suffices to have
$\liminf_{\substack{n\to\infty\\\gcd(n,q)=1}}\frac{a_f(n)a_g(n)}{n^{\frac{k+\ell}{2}-1}}\ge0$
for some fixed modulus $q$.
\item
However, infinitely many $n$ are required when $(M,k)=(N,\ell)$
and $\dim S_k^{\rm new}(\Gamma_0(M))>1$, i.e.\ the first sign
change of $a_f(n)a_g(n)$ cannot be effectively bounded.
To see this, let $\varphi,\psi\in S_k^{\rm new}(\Gamma_0(M))$ be distinct
normalized newforms, and set $f=\varphi+\pi\psi$,
$g_\varepsilon=f+\varepsilon\varphi$ for some $\varepsilon\ne0$.
Then $\min\{n\in\N:a_f(n)a_{g_\varepsilon}(n)<0\}\to\infty$ as $\varepsilon\to0$.
\end{enumerate}
\end{remarks}

Our proof makes use of the joint Sato--Tate equidistribution
of any two non-CM, twist-inequivalent newforms, proven by
Wong \cite{Wong}, based on the spectacular results
of Barnet-Lamb et al.\ \cite{BGG}. (See also Newton and Thorne
\cite{Newton-Thorne} for the recent strengthening
to full automorphy of symmetric powers for Hilbert modular forms
of regular weight, and Thorner \cite{Thorner} for results with an
effective rate of convergence.)
Precisely, let $\varphi\in S_k^{\rm new}(\Gamma_0(M))$,
$\psi\in S_\ell^{\rm new}(\Gamma_0(N))$
be distinct normalized Hecke eigenforms of conductors $M,N$ as in the theorem,
and write their Fourier coefficients at primes $p$ in the form
\[
a_\varphi(p)=2p^{\frac{k-1}{2}}\cos(2\pi\theta_\varphi(p)),
\quad
a_\psi(p)=2p^{\frac{\ell-1}{2}}\cos(2\pi\theta_\psi(p)),
\quad\text{where }
(\theta_\varphi(p),\theta_\psi(p))\in\bigl[0,\tfrac12\bigr]^2.
\]
Then by \cite[Theorem~1.1]{Wong}, for any box
$B=[\alpha,\beta]\times[\gamma,\delta]\subseteq[0,\frac12]^2$, we have
\[
\lim_{x\to\infty}
\frac{\#\{p\le x:(\theta_\varphi(p),\theta_\psi(p))\in B\}}{\pi(x)}
=\int_B\bigl(4\sin(2\pi u)\sin(2\pi v)\bigr)^2\,du\,dv.
\]
As we will see in Lemma~\ref{lem:linindep}, it follows that
$1$, $\theta_\varphi(p)$, and $\theta_\psi(p)$ are linearly independent
over $\Q$ for all $p$ in a set of density $1$.

We note that mutual independence of the angles $\theta_\varphi$ for
more than two newforms is not known, though it would follow from
the functoriality of arbitrary products of symmetric powers.
This presents an obstacle to proving Theorem~\ref{thm:main} using
only information on $a_f(p)a_g(p)$ at primes $p$.
(We will however give such a proof when $f$ and $g$ have different
weights or levels.) Instead we take an approach that is possibly of independent
interest, showing that for almost any tuple of
distinct primes $p_1,\ldots,p_d$, where $d=\dim S_k^{\rm new}(\Gamma_0(M))$,
the non-zero $f\in S_k^{\rm new}(\Gamma_0(M))$ with real coefficients
are determined modulo scalars by the signs of $a_f(n)$
for $n\in\{p_1^{k_1}\cdots p_d^{k_d}:k_1,\ldots,k_d\ge0\}$;
see Proposition~\ref{prop:dense} for
the precise statement, and \cite{Amri,GMP} for some related results.

\subsection*{Acknowledgements}
I thank Jonathan Bober and Oleksiy Klurman for
many helpful conversations.

\section{Different weight or level}
Note that the implication (2)$\implies$(1) is trivial,
so it suffices to prove the converse.
Let us first suppose that $(M,k)\ne(N,\ell)$,
for which we adapt the proof for Hecke eigenforms given in
\cite{KLSW}, based on Ramakrishnan's functorial lift
\cite{Ramakrishnan} from $\GL(2)\times\GL(2)$ to $\GL(4)$.
One could also extend the proof given in the next section
to allow $(M,k)\ne(N,\ell)$, at the expense of
complicating the notation, but the present proof
has the added feature that it only requires
$a_f(p)a_g(p)\ge0$ for a sufficiently dense set of primes $p$.

Let $\{\varphi_1,\ldots,\varphi_d\}$
and $\{\psi_1,\ldots,\psi_{d'}\}$ be normalized Hecke eigenbases
for $S_k^{\rm new}(\Gamma_0(M))$ and $S_\ell^{\rm new}(\Gamma_0(N))$, respectively.
Write
\[
\varphi_i(z)=\sum_{n=1}^\infty
\lambda_{\varphi_i}(n)n^{\frac{k-1}{2}}e(nz),
\quad
\psi_j(z)=\sum_{n=1}^\infty
\lambda_{\psi_j}(n)n^{\frac{\ell-1}{2}}e(nz),
\]
and
\[
\lambda_f(n)\coloneq\frac{a_f(n)}{n^{\frac{k-1}{2}}}
=\sum_{i=1}^d u_i\lambda_{\varphi_i}(n),
\quad
\lambda_g(n)\coloneq\frac{a_g(n)}{n^{\frac{\ell-1}{2}}}
=\sum_{j=1}^{d'} v_j\lambda_{\psi_j}(n).
\]
Our hypothesis on $\lcm(M,N)$ implies that
$\varphi_1,\ldots,\varphi_d,\psi_1,\ldots,\psi_{d'}$ are all
twist minimal without CM, and no two are twist equivalent.
By \cite[\S3, Theorem~M]{Ramakrishnan}, for any pair of indices $(i,j)$,
there is a cuspidal automorphic representation $\pi_{ij}$ of $\GL(4)$
with $p$th Dirichlet coefficient $\lambda_{\pi_{ij}}(p)$ satisfying
$\lambda_{\pi_{ij}}(p)=\lambda_{\varphi_i}(p)\lambda_{\psi_j}(p)$
for primes $p\nmid MN$.
Furthermore, following the proof of
\cite[Lemma~4.5.8]{Ramakrishnan}, one can see that
$\pi_{ij}\cong\pi_{i'j'}$ if and only if $(i,j)=(i',j')$.

Applying the Rankin--Selberg method (see \cite[Theorem~2.3]{Liu-Ye}), we have
\[
\sum_{p\le x}\frac{\lambda_f(p)\lambda_g(p)}{p}
=\sum_{i,j} u_iv_j\sum_{p\le x}
\frac{\lambda_{\varphi_i}(p)\lambda_{\psi_j}(p)}{p}
=O(1).
\]
On the other hand,
\begin{align*}
\sum_{p\le x}\frac{(\lambda_f(p)\lambda_g(p))^2}{p}
&=O(1)+\sum_{i,j,i',j'}u_iu_{i'}v_jv_{j'}\sum_{p\le x}
\frac{\lambda_{\pi_{ij}}(p)\lambda_{\pi_{i'j'}}(p)}{p}\\
&=O(1)+\sum_{i,j}u_i^2v_j^2\sum_{p\le x}
\frac{\lambda_{\pi_{ij}}(p)^2}{p}
=O(1)+\left(\sum_i u_i^2\right)\left(\sum_j v_j^2\right)\sum_{p\le x}\frac1p.
\end{align*}
However, from Deligne's bound we have
$|\lambda_f(p)\lambda_g(p)|\le4UV$, where $U=\sum_i|u_i|, V=\sum_j|v_j|$.
Applying this pointwise estimate and Cauchy--Schwarz, we have
\begin{align*}
\sum_{p\le x}\frac{(\lambda_f(p)\lambda_g(p))^2}{p}
&\le4UV\left(
\sum_{p\le x}\frac{\lambda_f(p)\lambda_g(p)}{p}
+8UV\sum_{\substack{p\le x\\\lambda_f(p)\lambda_g(p)<0}}\frac1p
\right)\\
&=O(1)+32U^2V^2\sum_{\substack{p\le x\\\lambda_f(p)\lambda_g(p)<0}}\frac1p
\le O(1)+32dd'\left(\sum_iu_i^2\right)\left(\sum_jv_j^2\right)
\sum_{\substack{p\le x\\\lambda_f(p)\lambda_g(p)<0}}\frac1p.
\end{align*}
Since $f$ and $g$ are non-zero by hypothesis, we have
$\bigl(\sum_iu_i^2\bigr)\bigl(\sum_jv_j^2\bigr)>0$, so we obtain
\[
\sum_{\substack{p\le x\\\lambda_f(p)\lambda_g(p)<0}}\frac1p
\ge\frac1{32dd'}\sum_{p\le x}\frac1p+O(1).
\]
This results in a contradiction if the upper Dirichlet density of
$\{p:\lambda_f(p)\lambda_g(p)\ge0\}$ exceeds
$1-\frac1{32dd'}$.

\section{Equal weight and level}
Hence we may assume that $(M,k)=(N,\ell)$.
With notation as above, write $a_i(n)=\lambda_{\varphi_i}(n)n^{\frac{k-1}{2}}$
for the Fourier coefficients of $\varphi_i$.
Let $K=\Q(\{a_i(n):i=1,\ldots,d,\,n\in\N\})$ be the coefficient
field, write $\OK$ for its ring of integers, and
define $\theta_i(p)\in[0,\frac12]$ for primes $p$ by
$a_i(p)=2p^{\frac{k-1}{2}}\cos(2\pi\theta_i(p))$.
We may assume that $d\ge2$, since there is nothing to prove
otherwise.

\begin{lemma}\label{lem:linindep}
There is a set $S\subseteq\{p\text{ prime}:p\nmid M\}$ of natural density $1$
such that, for any $p\in S$ and any pair $i\ne j$, the numbers
$1$, $\theta_i(p)$, and $\theta_j(p)$ are linearly independent
over $\Q$.
\end{lemma}
\begin{proof}
Consider linear relations of the form
$\sum_{i=1}^d n_i\theta_i(p)\in\Q$ for primes $p\nmid M$ and integers $n_i$.
Note that
\[
e(\theta_i(p))=\frac{a_i(p)+\sqrt{a_i(p)^2-4p^{k-1}}}{2p^{\frac{k-1}{2}}}
\in K\!\left(\sqrt{p},\sqrt{a_i(p)^2-4p^{k-1}}\right),
\]
so $[\Q(e(\theta_1(p)),\ldots,e(\theta_d(p))):\Q]\le 2^{d+1}[K:\Q]$.
Therefore if $\sum_{i=1}^d n_i\theta_i(p)$ is rational then its denominator
is bounded, so
\[
L=K\!\left(\left\{e\!\left(\sum_{i=1}^d n_i\theta_i(p)\right):
p\nmid M\text{ prime},\,
(n_1,\ldots,n_d)\in\Z^d,\,
\sum_{i=1}^d n_i\theta_i(p)\in\Q\right\}\right)
\]
is a finite extension of $K$. In particular,
if $\theta_i(p)$ is rational then either $a_i(p)=0$ or $\sqrt{p}\in L$.
The latter happens for at most finitely many $p$, while the former happens for
a density $0$ set of $p$, by \cite[p.~174, Corollaire~2]{Serre}.
Hence we may take $S$ to contain
only primes for which every $\theta_i(p)$ is irrational.

Next suppose $m\theta_i(p)\pm n\theta_j(p)\in\Q$ for some $i\ne j$
and coprime positive integers $m,n$. Then for some integer $q>0$, we have
\begin{equation}\label{eq:TmqTnq}
\cos(2\pi mq\theta_i(p))=\cos(2\pi nq\theta_j(p))
\implies
T_{mq}\!\left(\frac{a_i(p)}{2p^{\frac{k-1}{2}}}\right)
=T_{nq}\!\left(\frac{a_j(p)}{2p^{\frac{k-1}{2}}}\right)
\end{equation}
where $T_{mq},T_{nq}\in\Z[x]$ are Chebyshev polynomials.
We may assume without loss of generality that $q$ is even,
so both sides are elements of $K$.

If $a_i(p)/p^{\frac{k}{2}}$ is an algebraic integer then, taking the norm
from $K$ to $\Q$ and using Deligne's bound, we have
\[
\bigl|\Norm\bigl(a_i(p)/p^{\frac{k}{2}}\bigr)\bigr|
\le(2/\sqrt{p})^{[K:\Q]}.
\]
This is less than $1$ if $p\ge5$, which is a contradiction.
Hence if $p\nmid6\Disc(K)$ then there exists a prime ideal
$\p\mid p\OK$ such that $v_\p(a_i(p))\le\frac{k}{2}-1$.

Writing $T_{mq}(x)=\sum_{r=0}^{mq/2}c_r x^{2r}$, we have
\[
T_{mq}\!\left(\frac{a_i(p)}{2p^{\frac{k-1}{2}}}\right)
=\sum_{r=0}^{mq/2} c_r\frac{a_i(p)^{2r}}{(4p^{k-1})^r}.
\]
Since $p>2$ we have
\[
v_\p\left(c_r\frac{a_i(p)^{2r}}{(4p^{k-1})^r}\right)
=v_\p(c_r)+r\bigl(2v_\p(a_i(p))-(k-1)\bigr).
\]
Since the leading coefficient of $T_{mq}$ is a power of $2$
and $2v_\p(a_i(p))-(k-1)<0$, this is minimal when $r=mq/2$, so
\[
v_\p T_{mq}\!\left(\frac{a_i(p)}{2p^{\frac{k-1}{2}}}\right)
=\frac{mq}{2}\bigl(2v_\p(a_i(p))-(k-1)\bigr)<0.
\]

By \eqref{eq:TmqTnq},
$T_{nq}\!\left(\frac{a_j(p)}{2p^{\frac{k-1}{2}}}\right)$
also has negative valuation at $\p$, so we must have
$v_\p(a_j(p))\le\frac{k}{2}-1$,
and a similar calculation yields
\[
v_\p T_{nq}\!\left(\frac{a_j(p)}{2p^{\frac{k-1}{2}}}\right)
=\frac{nq}{2}\bigl(2v_\p(a_j(p))-(k-1)\bigr).
\]
Therefore
\[
m\bigl(k-1-2v_\p(a_i(p))\bigr)
=n\bigl(k-1-2v_\p(a_j(p))\bigr).
\]
Since $m$ and $n$ are coprime it follows that they are both odd
and less than $k$.

Since $\theta_i(p)$ and $\theta_j(p)$ are pairwise independent
and Sato--Tate distributed for $i\ne j$, the set of $p$ for which
$m\theta_i(p)\pm n\theta_j(p)\in\Q$ holds for some pair $i\ne j$
and some odd $m,n<k$ has density $0$.
This completes the proof.
\end{proof}

\begin{lemma}\label{lem:dense}
For any $p\in S$ and any index $i_1$, there are distinct
non-zero integers $n_1,\ldots,n_d$ such that
\begin{enumerate}
\item $n_{i_1}\nmid n_i$ for $i\ne i_1$;
\item for any $t\in\R$ and any $\varepsilon>0$, there exists
$n\in\N$ such that
\[
\min\{|n\theta_i(p)-n_it-m|:m\in\Z\}<\varepsilon
\quad\text{for }i=1,\ldots,d.
\]
\end{enumerate}
\end{lemma}
\begin{proof}
Fix $p\in S$ and an index $i_1$, and consider the vector space
\[
V=\Q+\Q\theta_1(p)+\cdots+\Q\theta_d(p).
\]
Write $\dim V=m+1\ge3$.
By permuting the indices if necessary, we may assume that
$\{1,\theta_1(p),\ldots,\theta_m(p)\}$ is a basis for $V$, and we may choose
the permutation to map $i_1$ to $1$.
For $i=1,\ldots,m$ we write
\[
\theta_i(p)=\sum_{j=1}^m\alpha_{ij}\theta_j(p) + \beta_i
\]
for some $\alpha_{ij},\beta_i\in\Q$.
By our construction, no two of the vectors
$(\alpha_{ij})_{j=1,\ldots,d}$ are colinear.

Let $q$ be a common denominator for the $\beta_i$, and for each $j$, let
$q_j$ be a common denominator for $\alpha_{ij}$. Then as $n$ ranges over $q\N$,
the vectors $(n\theta_1(p),\ldots,n\theta_m(p))$ are dense in
$\R/q_1\Z\times\cdots\times\R/q_m\Z$, so for any $(t_1,\ldots,t_m)\in\R^m$,
there exists $n$ such that
\[
n\theta_i(p)\approx\sum_{j=1}^m\alpha_{ij}t_j\pmod1
\quad\text{for }i=1,\ldots,d,
\]
with arbitrarily small error.

For $j=1,\ldots,m$ let $t_j=q_jb^{m-j}t$, where $b$ is an integer
satisfying $b>2q_j|\alpha_{ij}|$ for every $i,j$. Then the integers
$n_i=\sum_{j=1}^m\alpha_{ij}q_jb^{m-j}$
satisfy $0<|\frac{q_1n_i}{n_1}-q_1\alpha_{i1}|<1$ for $i>1$, and we may
find $n$ such that
\[
n\theta_i(p)\approx n_it\pmod1
\quad\text{for }i=1,\ldots,d.
\]
\end{proof}

\begin{proposition}\label{prop:dense}
Let $p_1,\ldots,p_d\in S$ be distinct primes. Then
\[
\bigl\{
\bigl[a_1\bigl(p_1^{k_1}\cdots p_d^{k_d}\bigr):\cdots:a_d\bigl(p_1^{k_1}\cdots p_d^{k_d}\bigr)\bigl]\in\RP^{d-1}:
k_1,\ldots,k_d\ge0
\bigl\}
\]
is dense in $\RP^{d-1}$.
\end{proposition}
\begin{proof}
Let $\theta_{ij}=\theta_j(p_i)$. For each $i$ we
apply Lemma~\ref{lem:dense} with $p=p_i$ and $i_1=i$,
and we denote the resulting integers $n_{i1},\ldots,n_{id}$.
Let $t=(t_1,\ldots,t_d)$, and
consider the vectors $F(t)=(F_j(t))_{j=1,\ldots,d}$, where
\[
F_j(t)=\prod_{i=1}^d\frac{\sin(2\pi n_{ij}t_i)}{\sin(2\pi\theta_{ij})}
=c_j\prod_{i=1}^d\sin(2\pi n_{ij}t_i),
\]
for some $c_j\in\R_{>0}$.
Then
\[
\frac{\partial F_j}{\partial t_i}
=2\pi c_jn_{ij}\cos(2\pi n_{ij}t_i)
\prod_{\substack{1\le i'\le d\\i'\ne i}}
\sin(2\pi n_{i'j}t_{i'}).
\]
By construction we have $n_{ii}\nmid n_{ij}$ for $j\ne i$,
and it follows that $F$ vanishes at the point
$t=(\frac1{2n_{ii}})_{i=1,\ldots,d}$, and
the Jacobian matrix there is diagonal with non-zero determinant.
By the inverse function theorem, the image of $F$ contains
an open neighborhood of the origin.

For any integer $k_i\ge0$ we have
\[
\lambda_{\varphi_j}(p_i^{k_i})=\frac{\sin(2\pi(k_i+1)\theta_{ij})}{\sin(2\pi\theta_{ij})}.
\]
By Lemma~\ref{lem:dense}, for any $t_i\in\R$, we can choose $k_i$ such that
\[
\lambda_{\varphi_j}(p_i^{k_i})\approx\frac{\sin(2\pi n_{ij}t_i)}{\sin(2\pi\theta_{ij})}
\quad\text{simultaneously for }j=1,\ldots,d,
\]
with arbitrarily small error. Therefore, the closure of
\[
\bigl\{\bigl(\lambda_{\varphi_1}(p_1^{k_1}\cdots p_d^{k_d}),\ldots,
\lambda_{\varphi_d}(p_1^{k_1}\cdots p_d^{k_d})\bigr)\in\R^d:
k_1,\ldots,k_d\ge 0\bigr\}
\]
contains the image of $F$.
\end{proof}

To complete the proof of Theorem~\ref{thm:main}, write
$f=\sum_{i=1}^d u_i\varphi_i$ and
$g=\sum_{i=1}^d v_i\varphi_i$
for some non-zero vectors $u=(u_1,\ldots,u_d),v=(v_1,\ldots,v_d)\in\R^d$.
If $f$ and $g$ are not proportional
then we may choose a vector $w$ such that
$w\cdot u=0$ and $w\cdot v\ne0$.
Then for any $\delta\in\R$, we have
\[
(w+\delta u)\cdot u=\delta(u\cdot u)
\quad\text{and}\quad
(w+\delta u)\cdot v=(w\cdot v)+\delta(u\cdot v).
\]
Choosing $\delta$ such that $\delta(w\cdot v)<0$ and
$\delta|u\cdot v|<|w\cdot v|$, we see that
$(w+\delta u)\cdot u$ and $(w+\delta u)\cdot v$ have different signs.
Applying Proposition~\ref{prop:dense} we can find
an arbitrarily large $n$ such that $[a_1(n):\cdots:a_d(n)]$ approximates
the image of $w+\delta u$ in $\RP^{d-1}$
arbitrarily closely, and it follows that
$a_f(n)$ and $a_g(n)$ have different signs.
This is a contradiction, so $f$ and $g$ must be proportional,
meaning that $c=f/g$ is constant.
Since $f$ and $g$ are non-zero and $a_f(n)a_g(n)\ge0$
for sufficiently large $n$, it follows that $c>0$.

\bibliographystyle{amsplain}
\bibliography{signs}
\end{document}